\documentclass[11pt]{amsart}
\usepackage{amsmath,amssymb,amsthm,graphicx,booktabs,hyperref,microtype,multicol,thmtools}
\newcommand{\abs}[1]{\left\lvert #1 \right\rvert}

\newcommand{\del}[1]{\left( #1 \right)}

\newcommand{\sbr}[1]{\left[ #1 \right]}

\newcommand{\RR}{\mathbb R}
\newcommand{\ZZ}{\mathbb Z}

\newtheorem{theorem}{Theorem}[section]
\newtheorem{proposition}[theorem]{Proposition}

\newtheorem{conjecture}[theorem]{Conjecture}
\newtheorem{lemma}[theorem]{Lemma}

\usepackage{amssymb}
\usepackage{amsthm}
\usepackage{amsmath}
\usepackage{amsbsy}
\usepackage[all]{xy}
\usepackage{bm}
\usepackage{hyperref}
\usepackage{array}
\usepackage{colortbl,xcolor}
\usepackage{ytableau}
\usepackage{hhline}
\usepackage{graphicx, wrapfig}
\usepackage[margin=1in]{geometry}

\allowdisplaybreaks
\usepackage{enumitem}

\usepackage{tikz-cd}

\newcommand{\Div}[1]{\mathop{\mathrm{Div}}\nolimits\del{#1}}
\renewcommand{\div}[1]{\mathop{\mathrm{div}}\nolimits\del{#1}}
\newcommand{\Divo}[1]{\mathop{\mathrm{Div}^0}\nolimits\del{#1}}

\newcommand{\jac}[1]{\mathop{\mathrm{Jac}}\nolimits\del{#1}}
\newcommand{\prin}[1]{\mathop{\mathrm{Prin}}\nolimits\del{#1}}
\newcommand{\im}[1]{\mathop{\mathrm{Im}}\nolimits\del{#1}}

\newcommand{\II}{\mathcal{I}}
\newcommand{\BB}{\mathcal{B}}

\newcommand{\sig}{\sigma}
\newcommand{\outdeg}{\text{outdeg}}

\title{Exponents of Jacobians of Graphs and Regular Matroids}
\author{Hahn Lheem}
\address{Harvard University, Cambridge, MA 02138}
\email{hahnlheem@college.harvard.edu}
\author{Deyuan Li}
\address{Yale University, New Haven, CT 06511}
\email{deyuan.li@yale.edu}
\author{Carl Joshua Quines}
\address{Massachusetts Institute of Technology, Cambridge, MA 02139}
\email{cjq@mit.edu}
\author{Jessica Zhang}
\address{Proof School, San Francisco, CA 94103}
\email{jessicajzhang03@gmail.com}

\begin{document}

\begin{abstract}
   Let $G$ be a finite undirected multigraph with no self-loops. The Jacobian $\jac G$ is a finite abelian group associated with $G$ whose cardinality is equal to the number of spanning trees of $G$. There are only a finite number of biconnected graphs $G$ such that the exponent of $\jac G$ equals $2$ or $3$. The definition of a Jacobian can also be extended to regular matroids as a generalization of graphs. We prove that there are finitely many connected regular matroids $M$ such that $\jac M$ has exponent $2$ and characterize all such matroids.
\end{abstract}

\maketitle

\section{Introduction}

Let $G$ be a finite undirected multigraph. The \emph{Jacobian} $\jac G$ is a group associated with $G$. It can be defined in several equivalent ways, hence why it is also known as the group of components, the critical group, the sandpile group, or the Smith group.

The \emph{wedge sum} of two graphs $G_1$ and $G_2$, denoted by $G_1 \wedge G_2$, is formed by identifying two vertices of the original graphs. It is known that $\jac{G_1 \wedge G_2} = \jac{G_1} \oplus \jac{G_2}$. As any connected graph can be written as the wedge sum of biconnected graphs \cite{bcc}, we are particularly interested in the structure of $\jac G$ when $G$ is biconnected. Recall that a graph is \emph{biconnected} if the induced subgraph formed by removing any vertex is connected; in particular, a graph that is not biconnected can be disconnected by removing a vertex.

We know that $\jac G$ is a finite abelian group, and its cardinality is equal to the number of spanning trees of $G$. We call a positive integer $m$ the \emph{exponent} of $\jac G$ if it is the smallest positive integer $m$ such that $ma=0$ for every $a\in\jac G$. We investigated the following conjecture, made by the proposer Matthew Baker and Farbod Shokrieh:

\begin{conjecture}\label{con:graphexp}
For every positive integer $k$, there are only finitely many biconnected graphs $G$ such that the exponent of $\jac G$ is at most $k$.
\end{conjecture}

In the case $k = 1$, the cardinality of $\jac G$ would be $1$, and there are only two biconnected graphs with exactly $1$ spanning tree. We prove the conjecture for the cases $k = 2$ and $k = 3$ using Dhar's burning algorithm.

The definition of a Jacobian can also be extended to regular matroids, which leads to the following analogous conjecture:

\begin{conjecture}\label{con:matroidexp}
For every positive integer $k$, there are only finitely many connected regular matroids $M$ such that the exponent of $\jac M$ is at most $k$.
\end{conjecture}

As all graphic matroids are regular matroids, this means that Conjecture~\ref{con:matroidexp} implies Conjecture~\ref{con:graphexp}. We prove Conjecture~\ref{con:matroidexp} for the case $k = 2$.

\section{Background}

\subsection{Divisor theory and the Jacobian}
\label{ssec:jacobian}
For the rest of this subsection, $G = (V, E)$ will represent a finite, undirected multigraph possibly with multi-edges but no self-loops. In analogy to divisors defined on Riemann surfaces, we can define divisors on graphs. A \emph{divisor} on $G$ is an integral linear combination of vertices, written as a formal sum
$$D = \sum_{v \in V} D(v)v.$$
The \emph{degree} of a divisor $D$ is
$$\deg(D) = \sum_{v \in V} D(v).$$
We distinguish $\deg(D)$ from the degree (valency) of a vertex $v \in V$ by using $\deg_G(v)$ for this instead.

Divisors can also be thought of as a configuration of sand grains on each vertex, where $D(v)$ counts the number of grains on a vertex $v$ if it is positive, and is a pit that can catch $-D(v)$ grains if it is negative. The divisors of a graph form an abelian group $\Div{G}$ under addition, of which the divisors with degree $0$ form a subgroup $\Divo{G}$.

Let $f$ be a function $f : V \to \ZZ$. Each such $f$ defines a divisor $\div{f}$, where $\div{f}(v)$ is defined as
$$\div{f}(v) = \sum_{\{v, w\} \in E(G)} (f(v) - f(w)) = \deg_G(v)f(v) - \sum_{\{v, w\} \in E} f(w).$$
Suppose that $D$ is a divisor for which there exists an $f$ such that $D = \div{f}$. Then $D$ is called a \emph{principal divisor}. We call two divisors $D_1$ and $D_2$ \emph{linearly equivalent}, denoted by $D_1 \sim D_2$, if their difference is principal.

Another way to think of linear equivalence is through \emph{sandpiles}. Going back to the sand grain analogy, we can \emph{topple} a vertex $v$ by removing $\deg_G(v)$ sand grains on the vertex and adding one sand grain to each of the neighbors of $v$. We can also topple in reverse by taking one sand grain from each of the neighbors of $v$ and adding $\deg_G(v)$ sand grains on $v$.

Then two divisors are linearly equivalent if one can be obtained from other through a series of topples. In particular, if $f : V \to \ZZ$, then we get the divisor $\div{f}$ if, starting with the zero divisor, vertex $v$ is toppled $f(v)$ times. When $f(v)$ is negative, then we topple in reverse that many times. Similarly, if $D_1 \sim D_2$, there exists some $f$ such that $D_1 + \div{f} = D_2$. This can be interpreted as $D_2$ being the result when, starting with $D_1$, vertex $v$ is toppled $f(v)$ times. Equivalently, since $\div{f}$ is a principal divisor, $D_1 \sim D_2$ if their difference is principal.

The set of all principal divisors is denoted by $\prin{G}$, which is a subgroup of $\Divo{G}$. Then $\jac{G}$ is defined as $\Divo{G}/\prin{G}$. The Jacobian of a connected graph is always a finite abelian group, with order equal to the number of spanning trees in $G$ (see for example \cite{big97}).

\subsection{The cycle space and cut space}
\label{ssec:cycle_cut}
This definition of the Jacobian, while easy to visualize, cannot be directly generalized to regular matroids, which do not have a concept of vertices. We then describe a definition of the Jacobian using solely the edges, cycles, and cuts of $G$, as defined in \cite{big97}.

Let $n = \abs{V}$ and $m = \abs{E}$. We will arbitrarily orient each of the edges of $G$. For each edge $e=\{v,w\}$, we pick one of $v$ and $w$ to be the \emph{head} of $e$, denoted $h(e)$. The other vertex incident to $e$ is called the \emph{tail} of $e$ and is denoted $t(e)$. This defines an \emph{orientation} of $G$. 

The \emph{incidence matrix} $D=(d_{ve})$ is defined to be the $n\times m$ matrix of $G$, given by 
\[d_{ve}=\begin{cases}1&\text{if }v=h(e)\\-1&\text{if }v=t(e)\\0&\text{otherwise}\end{cases}.\]

Denote by $C^1(G,\RR)$ the vector space of functions $f:E\to\RR$ with inner product given by \[\langle x,y\rangle=\sum_{e\in E}x(e)y(e).\] This is known as the \emph{edge space}, and we will abbreviate $C^1(G,\RR)$ as $C^1$ when the context is clear. Observe that the incidence matrix $D$ is a function on $C^1$.

Let $Z=\ker D\subseteq C^1$, also called the \emph{cycle space} of $G$. Let $B=Z^\perp$ be the orthogonal complement of $Z$ under the inner product defined above, also called the \emph{cut space}. These spaces are all well-defined (i.e. independent of the choice of orientation) up to isomorphisms, and it follows from this definition that \[C^1(G,\RR)=Z\oplus B.\]

The names cycle space and cut space come from the following interpretation. For any \emph{cycle} $Q$, consider the unique function $z_Q$ on $E$ that takes an edge $e$ to $1$ if it is on the cycle and is oriented in the same direction as it, $-1$ if it oriented in the opposite direction, and $0$ otherwise. Formally, a cycle $z_Q \in C^1$ is the function
\[
z_Q(e) = \begin{cases}
1 & \text{if } t(e), e, h(e) \text{ is in $Q$, in that order} \\
-1 & \text{if } h(e), e, t(e) \text{ is in $Q$, in that order} \\
0 & \text{otherwise}
\end{cases}.
\]
We can think of $z_Q$ as a signed characteristic function for the cycle $Q$. This is defined such that $D(z_Q) = 0$ for any cycle $z_Q$. It can be shown that the \emph{cycle space} $Z=\ker D$ consists of all linear combinations of cycles.

Suppose $U$ is a nonempty proper subset of $V$. A \emph{cut} on $U$ is a function on $E$ that takes an edge $e$ to $1$ if only its head is in $U$, $-1$ if only its tail is in $U$, and $0$ otherwise. More formally, a cut $b_U\in C^1$ is the function \[
b_U(e) = \begin{cases}
1 & \text{if $h(e)$ is in $U$ and $t(e)$ is not} \\
-1 & \text{if $t(e)$ is in $U$ and $h(e)$ is not} \\ 
0 & \text{otherwise}
\end{cases}.\]
This function is a signed characteristic function for the cut determined by $U$. The \emph{cut space} $B$ is formed from linear combinations of cuts $b_U$. It is a well-known theorem that this is identical to the previous definition of $B$ (see for example \cite{big97}). In particular, we know that $Z$ and $B$ are orthogonal complements.

Because of this, any element in $C^1$ can be written as a sum of an element in $Z$ and an element in $B$. In the following figure, an element $c \in C^1$ is on the left; to its right are $z_c \in Z$ and $b_c \in B$ such that $c = z_c + b_c$. Below these show the decomposition of $z_c$ as a sum of cycles and $b_c$ as a sum of cuts.

\begin{center}
    \includegraphics[width = 0.8 \textwidth]{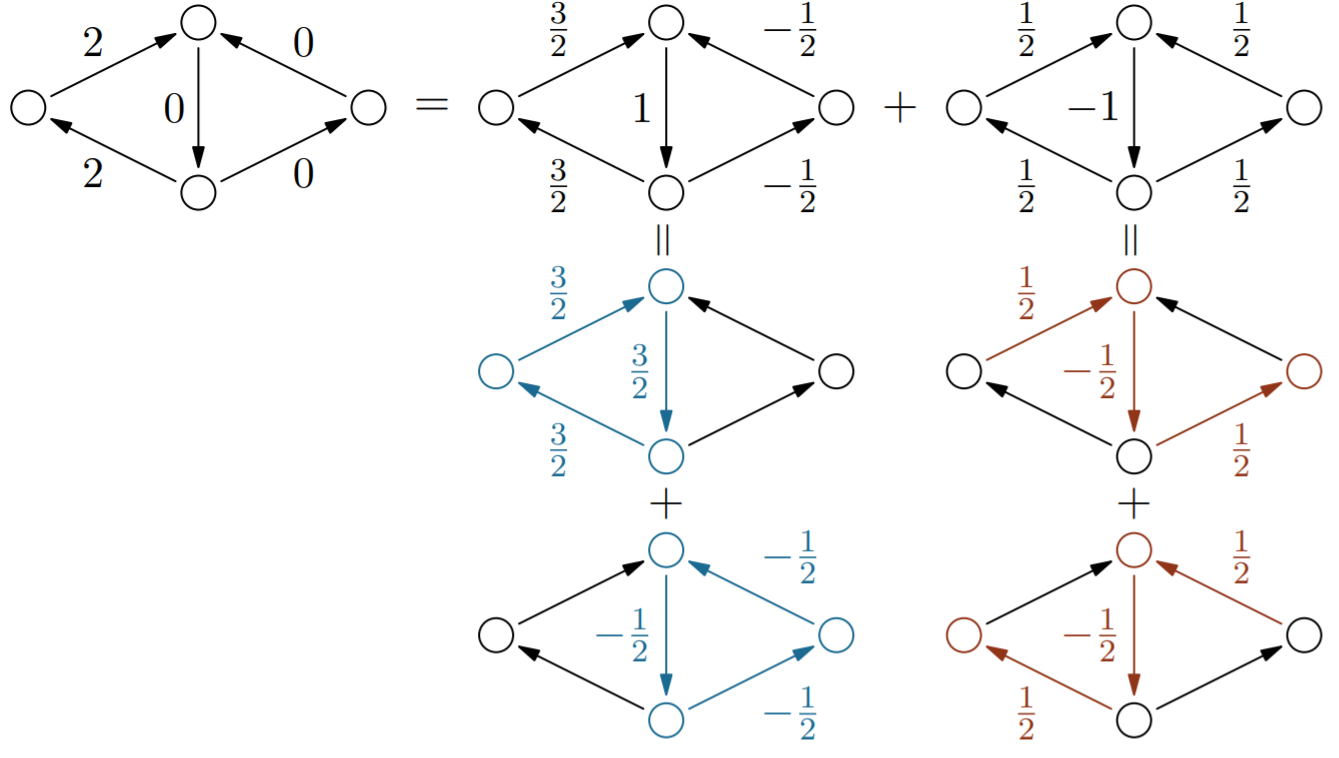}
\end{center}

We can also define the lattices of each of these spaces, consisting of the elements where all coordinates are integers. Let $C^1(G, \ZZ) = C_I$ be the \emph{edge lattice}, $Z_I$ be the \emph{cycle lattice}, and $B_I$ be the \emph{cut lattice}.

While it is true that every element in $C^1$ can be written as a sum of elements in $Z$ and $B$, it is not true that every element in $C_I$ is a sum of elements in $Z_I$ and $B_I$, as the previous example shows. However, it is true that $Z_I \oplus B_I$ is a sublattice of $C_I$.

\subsection{The projection matrix}
\label{ssec:proj}
As mentioned previously, because $C^1=Z\oplus B$, it follows that for any $c\in C^1$, we can find $z_c\in Z$ and $b_c\in B$ such that $c=z_c+b_c$. Let $P$ be the orthogonal projection $P:C^1\to B$ taking $c$ to $b_c$, which we will call the \emph{projection matrix}. From linear algebra, we know that $P$ is a linear transformation; its properties will be crucial to the proof of the conjecture for regular matroids in the case $k = 2$. 

There is a canonical factorization of $P$ given by $P=XD$, where $D$ is the incidence matrix and $X$ is an isomorphism from $\im D$ to $\im{D^t}$. To see this, observe that $\ker D=Z$ implies that $\im{D^t}=Z^\perp=B$. It follows from this that the diagram 
\[
\begin{tikzcd}
    C^1 \arrow{r}{P} \arrow[swap]{d}{D} & B \arrow{d}{=} \\
    \im D \arrow{r}{X} & \im{D^t}
\end{tikzcd}
\] 
is commutative. But because $\im P$ and $\im D$ have the same dimension, it follows that $X$ is an isomorphism between $\im D$ and $\im{D^t}$ (see for example \cite{big97}).

Recall from Section~\ref{ssec:cycle_cut} that although $C^1=Z\oplus B$, it is not the case that $C_I = Z_I\oplus B_I$. However, note that $C^1/(Z\oplus B)$ and $\im P/B$ are isomorphic; indeed, they are both the same as the trivial group. Similarly, if we let $P_I$ be the restriction of $P$ onto the edge lattice $C_I$, then there is an isomorphism induced by $P$, namely \[P_*:\frac{C_I}{Z_I\oplus B_I}\xrightarrow{\sim}\frac{\im{P_I}}{B_I},\] which is given by mapping the coset $[c]$ to the coset $[Pc]$. 

To see this, observe that a function $c\in C_I$ is in $Z_I\oplus B_I$ if and only if $Pc\in B_I$. In particular, if $Pc\in B_I\subset B$, then $c-Pc\in Z$ must have integer coordinates and is in $Z_I$. Thus $c\in Z_I\oplus B_I$. Conversely, if $c\in Z_I\oplus B_I$, then because $c=(c-Pc)+Pc$ is a decomposition of $c$, it follows that $Pc\in B_I$. 

This fact implies that the function taking $[c]$ to $[Pc]$ must be injective. That it is surjective follows from the fact that for any element $[Pc]\in\im{P_I}/B$, we know that $[c]\in P_*^{-1}([Pc])$. Finally, notice that \[P_*([c_1]+[c_2])=\sbr{(c_1+c_2)}=P_*(c_1+c_2)\] implies that $P_*$ is an isomorphism.

We will use without proof the fact that the image of $P_I$ is simply the \emph{dual lattice} $B_I^\#$ of $B_I$, which is given by \[B_I^\#=\{x\in B:\langle x,b\rangle\in\ZZ\text{ for all }b\in B_I\}.\]

It turns out that \[\jac G=\Divo G/\prin G=D(C_I)/D(B_I) \cong B_I^\#/B_I.\] For a proof of this see Appendix~\ref{a:equiv}. The alternate definitions of the Jacobian will generalize more readily to regular matroids. For the remainder of this paper, we will use the various definitions of the Jacobian interchangeably. 

\subsection{Matroid theory}

A finite \emph{matroid} $M$ is a pair $(E,\II)$, where the \emph{ground set} $E$ is a finite set and the \emph{independent sets} $\II$ is a family of subsets of $E$, that satisfies the following properties:
\begin{enumerate}
    \item The empty set is independent.
    \item If $S \in \II$ and $S_1 \subset S$, then $S_1 \in \II$.
    \item If $S, T \in \II$ and $|S| > |T|$, then there exists $s \in S\setminus T$ such that $T \cup \{s\} \in \II$.
\end{enumerate}

We define the set of \emph{bases} $\mathcal B$ of a matroid to be the set of maximal independent sets of $\mathcal I$ and define the set of \emph{circuits} $\mathcal C$ to be the minimal dependent sets.

Matroids can be thought of as generalizations of graphs. In particular, given any graph $G$, there is a corresponding matroid, namely the matroid whose ground set is the set of edges and whose bases are the spanning forests of $G$. Such a matroid is known as a \emph{graphic matroid}. Observe that the independent sets of a graphic matroid are realizable as the subforests of the graph. 

Matroids also generalize the notion of matrices. A \emph{linear matroid} is a matroid derived from a matrix over a given field. Its ground set $E$ is the set of column vectors of the matrix and $\II$ be the set of linearly independent elements of $E$. We call a matroid \emph{regular} if it can be represented as a linear matroid over all fields. This is equivalent to it having a representation over $\RR$ as a \emph{totally unimodular matrix}, which is a matrix where the determinant of every square submatrix is either $-1$, $0$, or $1$.

Consider an oriented incidence matrix corresponding to a graph. It can be shown that the linear matroid defined by the incidence matrix is isomorphic to the graphic matroid defined by the graph, regardless of on which field the matrix is defined. Thus all graphic matroids are regular matroids. On the other hand, there exist regular matroids that are not graphic \cite{oxley}.

Suppose a regular matroid $M$ is represented over $\RR$ by a totally unimodular matrix $D$. We define $C^1 = \RR^E$ as its \emph{edge space}, $Z = \ker D$ as its \emph{cycle space}, and $B = Z^{\perp}$ as the \emph{cut space}, in analogy to the definitions for graphs. We similarly define the \emph{edge lattice} $C_I$, the \emph{cycle lattice} $Z_I$, and the \emph{cut lattice} $B_I$, as well as the \emph{projection matrix} $P$. We then define the matroid's \emph{Jacobian} to be
$$\jac{M} = D(C_I)/D(B_I) \cong B_I^{\#}/B_I.$$
We know these are isomorphic through the earlier proofs, which do not depend on the structure of the graph $G$, but on $C^1$, $Z$, $B$, and so on, instead. Hence these definitions match the ones for graphs when the matroid is graphic.

\section{Main results}

\subsection{Dhar's burning algorithm}\label{ssec:dharsburning}

We will prove Conjecture~\ref{con:graphexp} for the cases $k = 2,3$. To do this, we will use the theory of $q$-reduced divisors.

Given a divisor $D$ for some graph, $D$ is a \emph{$q$-reduced divisor} for some vertex $q$ if the graph satisfies the following properties:
\begin{enumerate}
    \item For all vertices $v \neq q, f(v)\ge 0$.
    \item For all nonempty $A\subseteq V\setminus \{q\}$, there exists a $v \in A$ such that $\outdeg_G(v) > f(v)$ where $\outdeg_G(v)$ is the number of edges connecting $v$ to a vertex not in $A$.
\end{enumerate}

Recall from Section~\ref{ssec:jacobian} that $\jac{G} = \Divo{G}/\prin{G}$ where $\prin{G}$ is the group of all principal divisors. Thus, $\jac{G}$ is the set of all equivalence classes of divisors of $G$, where the equivalence relation is linear equivalence. But we know that every element of $\Divo{G}$ is equivalent to exactly one $q$-reduced divisor \cite{bn07}. Thus there is a bijection between the set of $q$-reduced divisors and the elements of $\jac G$.

\emph{Dhar's burning algorithm} allows us to determine whether a given divisor is a $q$-reduced divisor. First, we burn the vertex at $q$. At each step, a vertex $v$ burns if $f(v)$ is less than the number of edges between $v$ and a previously burned vertex. If the entire graph ultimately burns, then Dhar's burning algorithm implies that the original divisor was a $q$-reduced divisor.

\begin{lemma}
    The exponent of $\jac{G}$ is greater than or equal to the maximum degree of a vertex of $G$. 
\end{lemma}

\begin{proof}
    Let $v$ be a vertex of maximal degree in $G$ and let $q$ be any other vertex. Then for each $i\in\{0,1,\dots,\deg_G(v)-1\}$, consider the divisor $D_i=i(v)-i(q)$.

    By Dhar's burning algorithm, each of these $\deg_G(v)$ divisors are unique $q$-reduced divisors. As a result, each divisor $D_i$ corresponds to a unique element of $\jac{G}$. 

    Since $D_i=iD_1$, and none of $D_0,D_1, D_2, \ldots ,D_{\deg_G(v)-1}$ are equivalent, the order of $D_1$ is greater than or equal to $\deg_G(v)$ and hence the exponent of $\jac G$ is also greater than or equal to $\deg_G(v)$.
\end{proof}

It follows that for a biconnected graph to have exponent $2$, the degree of any vertex must be $2$. Thus, it must be a cycle. However, the exponent of $C_n$ is equal to $n$, so the only graph with exponent $2$ is $C_2$.

Furthermore, \cite[Lemma 29]{gjr14} tells us that if the exponent of $\jac{G}$ is equal to the maximum degree of a vertex of a biconnected graph $G$, then it must be a banana graph, a graph with two vertices and some number of edges between them. Thus for the case of $k=3$, we know that either the maximum degree of a vertex is three in which case it is the banana graph with three edges, or the maximum degree of a vertex is less than or equal to two which means it is a $3$-cycle. 

\subsection{Extension to regular matroids}\label{ssec:matroidsproof}

We now prove Conjecture~\ref{con:matroidexp} for the case $k = 2$. Let $M$ be a connected regular matroid with ground set $E$, represented over $\RR$ by a totally unimodular matrix $D$. The key idea in the proof will be to use the properties of the projection matrix $P$. Recall that $P$ is a map from $C^1 \to B_I^\#$, and let $n=|E|$, so that $P$ is an $n \times n$ matrix.

Let $e_i$ be a vector in $C^1$ that is $1$ at the $i$th coordinate, with all other coordinates equal to $0$. Note that $De_i$ is just the $i$th column of $D$, and hence may be identified with an element of $E$.

We first begin by characterizing the entries of $P$:

\begin{lemma}\label{lem:denominatortwo}
If $\jac M$ has exponent $2$, then the entries of $P$ are either $-\frac{1}{2}$, $0$, or $\frac{1}{2}$.
\end{lemma}

\begin{proof}
We will use the fact that the entries of $P$ are given by the explicit formula $\frac1\kappa\sum_{B\in\BB}N_B$ where $\kappa$ is the number of bases and for each $B\in\BB$, the matrix $N_B$ is a specific matrix whose entries are all $0$, $1$, or $-1$, from which it follows that the entries of $P$ are between $-1$ and $1$ and that $P$ is symmetric \cite{maurer}.

Since $\jac M$ has exponent at most $2$, every element $x\in\jac M$ must satisfy $2x=0$. Because $\jac M=B_I^\#/B_I$, we know that $2b\in B_I$ for any element $b\in B_I^\#$. By considering $B_I^\#$ to be a subgroup of $\RR^n$, it follows that each element of $B_I^\#$ must have coordinates with denominator at most $2$.

Recall now that $P(C_I)=B_I^\#$. For each $i$, we know that $e_i\in C_I$, so $P(e_i)\in B_I^\#$. Thus the coordinates of $P(e_i)$ have denominator at most $2$.

But $P(e_i)$ is simply the $i$th column of the projection matrix $P$. Thus we conclude that the entries of $P$ are either $-\frac12$, $0$, or $\frac12$.
\end{proof}

We then characterize $P$ even further, using the fact that $P$ is a projection matrix:

\begin{lemma}\label{lem:twononzero}
Each diagonal entry of $P$ is $\frac12$, and each row has exactly two nonzero entries.
\end{lemma}

\begin{proof}
We first prove that the diagonal entries are $\frac12$. As $P$ is a projection, we see $P^2 = P$. This is because all vectors in the image of $P$ are already in the cut space, so projecting them again to the cut space would result in the same thing. From here, it follows from the fact that $P$ is symmetric and matrix multiplication that
\begin{equation*}
    P_{i, 1}^2 + P_{i, 2}^2 + \cdots + P_{i, i}^2 + \cdots + P_{i, n}^2 = P_{i, i}. \tag{$\star$}
\end{equation*}
Thus $P_{i, i} \ge 0$.

If $P_{i, i} = 0$, then all the entries in the $i$th column are $0$, and hence $Pe_i = 0$ as well. Thus $e_i \in \ker P$. As $P$ is a projection to the cut space $B$, it follows $e_i$ is orthogonal to $B$, so $e_i \in Z$. But $Z = \ker D$, so $De_i = 0$, implying that the $i$th column of $D$ is $0$. But any set of linearly independent vectors cannot include zero which contradicts $M$ being loopless, so by Lemma~\ref{lem:denominatortwo}, $P_{i, i} = \frac12$.

We then show that there is exactly one $j \neq i$ such that $\abs{P_{i, j}} = \frac12$. Indeed, because $P_{i, i} = \frac12$, substituting into $(\star)$ shows that
$$P_{i, 1}^2 + P_{i, 2}^2 + \cdots + P_{i, i-1}^2 + P_{i, i+1}^2 + \cdots + P_{i, n}^2 = \frac14.$$
But again by Lemma~\ref{lem:denominatortwo}, each of the $P_{i, j}$s is either $-\frac12$, $0$, or $\frac12$. Since the sum of their squares is $\frac14$, it follows that at most one of these is $-\frac12$ or $\frac12$, and the rest of the entries are $0$, finishing the proof.
\end{proof}

We now finish the proof of Conjecture~\ref{con:matroidexp} for $k = 2$ from Lemma~\ref{lem:denominatortwo} and Lemma~\ref{lem:twononzero}:

\begin{proof}[Proof of Conjecture~\ref{con:matroidexp} for $k=2$]
Without loss of generality, choose $i < j$ such that $P_{i, j} = -\frac12$. By symmetry of $P$, $P_{j, i} = -\frac12$. Then
$$P = \begin{bmatrix}      & \vdots   &        & \vdots   &        \\
     \cdots & \frac12  & \cdots & -\frac12 & \cdots \\
            & \vdots   &        & \vdots   &        \\
     \cdots & -\frac12 & \cdots & \frac12  & \cdots \\
            & \vdots   &        & \vdots   &        \end{bmatrix}.$$
The indicated columns are the $i$th and $j$th columns, and all other entries in these columns are zero. Similarly, the $i$th and $j$th rows are shown, and all other entries in these rows are zero as well. Thus $P(e_i + e_j) = 0$, and through similar logic in the proof of Lemma~\ref{lem:twononzero}, $D(e_i + e_j) = 0$, and hence the $i$th and $j$th elements of $E$ form a circuit.

As all the other entries of the $i$th row are zero, then this $i$th element is not contained in any other circuit as well, and similarly for the $j$th element. As $M$ is connected, there can be no other elements in the matroid. 

Thus either $M$ contains one element or two elements, and there are only finitely many possibilities.
\end{proof}

Note that this proof does not rely on the fact that $M$ is connected up until the end. In fact, we can modify this proof to get a full characterization of all regular matroids and graphs whose Jacobians have exponent $2$. In particular, we see that any graph with exponent $2$ must be a tree with some edges that are doubled.

\section{Future work}

The approach in Section~\ref{ssec:dharsburning} does not seem easy to generalize beyond the case $k = 3$. In contrast, the approach detailed in Section~\ref{ssec:matroidsproof} does appear to be more readily generalizable.

For example, when $k = 3$, the proof in Lemma~\ref{lem:denominatortwo} can be adapted to show that the entries of $P$ are $0, \pm\frac{1}{3}, \pm\frac{2}{3}$. In a similar manner, Lemma~\ref{lem:twononzero} can be adapted to show that the diagonal entries are $\frac{1}{3}$ or $\frac{2}{3}$, each row and column has exactly three nonzero entries, and all off-diagonal entries have to be $\pm\frac{1}{3}$. However, the authors do not immediately see how this leads to a characterization of the matroid $M$.

\section*{Acknowledgements}

The authors would like to thank Matthew Baker and Elizabeth Pratt for their constant guidance and support throughout our research. We especially thank Matthew Baker and Farbod Shokrieh for their idea in the graphical case from which the proof of Conjecture~\ref{con:matroidexp} for the case $k=2$ was based on. We would also like to thank the PROMYS program and the Clay Mathematics Institute for giving us this research opportunity. Finally, we appreciate David Fried and Roger Van Peski for organizing this research opportunity in the PROMYS program. This paper also benefited from the suggestions of an anonymous referee.

\appendix

\section{Equivalence of Jacobian definitions}
\label{a:equiv}

We stated without proof in Section~\ref{ssec:proj} that if we define $\jac G$ to be the quotient $\Divo G/\prin G$, then \[\jac G=C_I/(Z_I\oplus B_I)\cong B_I^\#/B_I.\] We prove this statement now.

Let $C^0(G,\ZZ)$ be the space of integer-valued functions on $V$, which we will abbreviate as $C^0$. Then consider the integer-valued function $\sig:C^0\to\ZZ$ defined by \[\sig(\psi)=\sum_v\psi(v).\] Let $D_I$ be the restriction of $D$ to $C_I$ so that $D_I:C_I\to C^0$. Then it has been shown that the sequence
\[
\begin{tikzcd}
    C_I \arrow{r}{D_I} & C^0 \arrow{r}{\sig} & \ZZ \arrow{r}{} & 0
\end{tikzcd},
\] 
is exact \cite{big97}. In particular, $D(C_I)=\im{D_I}=\ker\sig$. 

\begin{proposition}
The Jacobian of a graph can also be written as $D(C_I)/D(B_I)$. That is, \[\Divo G/\prin G=D(C_I)/D(B_I).\]
\end{proposition}

\begin{proof}
Note the similarity between $D(C_I)=\ker\sig$ and $\Divo G$. In particular, we know that $\ker\sig$ is the set of integer-valued functions $\psi$ on $V$ such that $\sum_v\psi(v)=0$. On the other hand, $\Divo G$ is the set of divisors, which are simply integer-valued functions on $V$, with degree $0$. From this, it follows that $D(C_I) = \Divo G$.

Now consider the effect of $D$ on a function $f\in C_I$. It takes $f$ to the function $Df$ given by \[(Df)(v)=\sum_{h(e)=v}f(e)-\sum_{t(e)=v}f(e).\] Let $A_v$ denote the set of vertices adjacent to $v$. Then we claim that $(Df)(v)$ is of the form \[(\div g)(v)=\sum_{w\in A_v}\del{g(v)-g(w)},\] precisely when $f\in B_I$. 

If $f\in B_I$, then there exists an integer-valued function $g$ on the vertices such that $f(e)=\sum_{w\in V}g(w)b_w(e)$ where $b_w(e)$ is the cut determined by the vertex $w$ as explained in Section~\ref{ssec:cycle_cut}. But then \[(Df)(v)=\sum_{h(e)=v}\sum_{w\in V}g(w)b_w(e)-\sum_{t(e)=v}\sum_{w\in V}g(w)b_w(e).\] In the first term, the definition of $b_w$ implies that $b_w(e)=1$ when $w=v$. If $h(e)=v$ and $t(e)=w$, then $b_w(e)=-1$. Otherwise, $b_w(e)=0$. Thus, the first term can be rewritten as \[\sum_{h(e)=v}g(v)-\sum_{h(e)=v}\sum_{t(e)=w}g(w).\] Similarly, we find that the second term can be rewritten as \[\sum_{t(e)=v}g(v)-\sum_{t(e)=v}\sum_{h(e)=w}g(w).\] It follows, then, that \[(Df)(v)=\sum_{e=vw}g(v)-\sum_{w\in A_v}g(w)=\sum_{w\in A_v}(g(v)-g(w))=(\div g)(v),\] which was what we wanted.

Conversely, suppose that $(Df)(v)=(\div g)(v)$ for some $g\in C^0(G,\ZZ)$. Then by reversing the previous argument, we see that \[f(e)=\sum_{w\in V}g(w)b_w(e).\] Because $g$ and $b_w$ are both integer valued, so too is $f$. Moreover, by definition of $B$, we know that $f\in B$, so $f\in B_I$.

From this, it follows that $D(B_I)=\prin G$. Since $D(C_I)=\Divo G$, it follows that the Jacobian $\Divo G/\prin G$ can also be written as $D(C_I)/D(B_I)$, which was precisely what we wanted to show.
\end{proof}

With this alternate definition of the Jacobian in hand, we are able to prove the following proposition:

\begin{proposition}
Define $X_*:D(C_I)/D(B_I)\to B_I^\#/B_I$ to be the map taking the $[Dc]$ to $[Pc]$ for each $c\in C_I$. Then $X_*$ is an isomorphism. In particular, this implies that $\jac G = B_I^\#/B_I$.
\end{proposition}

\begin{proof}
We follow the proof presented in \cite{big97}.

We know already that the orthogonal projection $P:C^1\to B$ has a factorization $P=XD$ where $X$ is an isomorphism from $\im D$ to $\im{D^t}$ and $D$ is the incidence matrix. We also have that $\im{P_I}=B_I^\#$. This implies that \[X(D(C_I))=P(C_I)=B_I^\#.\] Since $P_I$ is the identity on $B_I$, we find that \[X(D(B_I))=P(B_I)=B_I.\] Thus the induced homomorphism $X_*$ taking $[Dc]\in D(C_I)/D(B_I)$ to $[XDc]=[Pc]\in B_I^\#/B_I$ is a well-defined mapping. After all, if $[Dc_1]=[Dc_2]$ for $c_1,c_2\in C_I$, then there exists some $b\in B_I$ so that $c_1=c_2+b$. But because $Pb\in B_I$ and $Pc_1=Pc_2+Pb$, it follows that $[Pc_1]=[Pc_2]$. 

Similarly, we can check that the map $D_*$ from $C_I/(Z_I\oplus B_I)$ to $D(C_I)/D(B_I)$ defined by $D_*([c])=[Dc]$ is a well-defined homomorphism. Moreover, the map $P_*$ taking $[c]\in C_I/(Z_I\oplus B_I)$ to $[Pc]\in B_I^\#/B_I$ is in fact an isomorphism, as shown in Section~\ref{ssec:proj}. Since $P_*=X_*D_*$, the following diagram commutes:
\[
\begin{tikzcd}
    C_I/(Z_I\oplus B_I) \arrow[swap]{d}{D_*} \arrow{dr}{P_*} \\
    D(C_I)/D(B_I) \arrow[swap]{r}{X_*} & B_I^\#/B_I.
\end{tikzcd}
\]
But because $P_*$ is an isomorphism, it follows that $D_*$ and $X_*$ must also be isomorphisms. So $D(C_I)/D(B_I)$ and $B_I^\#/B_I$ are isomorphic.
\end{proof}

It thus follows that the three definitions of $\jac G$ give rise to isomorphic groups and are therefore consistent.

\end{document}